\documentclass[11pt,oneside,a4paper]{amsart}

\usepackage{lmodern}
\usepackage{color}
\usepackage{amssymb}

\newtheorem{thm}{Theorem}
\newtheorem{prop}[thm]{Proposition}
\newtheorem{lem}[thm]{Lemma}

\theoremstyle{remark}
\newtheorem{exam}[thm]{Example}

\newcommand{\barre}[1]{\overline{#1}}

\newcommand{\p}{\partial}
\newcommand{\R}{\mathbb{R}}
\newcommand{\B}{\mathbb{B}}
\newcommand{\boC}{\mathcal{C}}
\newcommand{\boM}{\mathcal{M}}
\renewcommand{\S}{\mathbb{S}}

\usepackage[left=2.5cm,right=2.5cm,top=2.5cm,bottom=2.5cm]{geometry}

\title{Free boundary minimal hypersurfaces outside of the ball}

\author{Laurent Mazet}
\address{Institut Denis Poisson, CNRS UMR 7013, Universit\'e de Tours, 
Universit\'e d'Orl\'eans, Parc de Grandmont, 37200 Tours, France}
\email{laurent.mazet@univ-tours.fr}
\author{Abra\~ao Mendes}
\address{Instituto de Matem\'atica, Universidade Federal de Alagoas, Macei\'o, AL, Brazil}
\email{abraao.mendes@im.ufal.br}
\thanks{L.M. was partially supported by the ANR-19-CE40-0014 grant. A.M. was partially supported by the National Council for
Scientific and Technological Development – CNPq, Brazil (Grant 305710/2020-6). The authors were partially supported by Coordena\c{c}\~ao de Aperfei\c{c}oamento de Pessoal de N\'ivel Superior – Brasil (CAPES-COFECUB 88887.143161/2017-0).}

\numberwithin{equation}{section}

\DeclareMathOperator{\Vol}{Vol}
\DeclareMathOperator{\Ind}{Ind}
\DeclareMathOperator{\Ric}{Ric}
\DeclareMathOperator{\divergent}{div}

\newcommand{\boL}{\mathcal{L}}
\newcommand{\Ome}{\Omega}
\newcommand{\eps}{\varepsilon}
\renewcommand{\div}{\divergent}

\usepackage{setspace}

\usepackage{marginnote}

\usepackage{graphicx}
\usepackage{float}

\begin{document}

\begin{abstract}
In this paper we obtain two classification theorems for free boundary minimal 
hypersurfaces outside of the unit ball (exterior FBMH for short) in 
Euclidean space. The first result states that the only exterior stable FBMH with parallel 
embedded regular ends are the catenoidal hypersurfaces. To achieve this we prove a 
B\^ocher type result for positive Jacobi functions on regular minimal ends in 
$\mathbb{R}^{n+1}$ which, after some calculations, implies the first theorem. The second 
theorem states that any exterior FBMH $\Sigma$ with one regular end is a catenoidal 
hypersurface. Its proof is based on a symmetrization procedure similar to 
R.~Schoen~\cite{Sch2}. We also give a complete description of the catenoidal hypersurfaces, 
including the calculation of their indices.
\end{abstract}

\maketitle

\section{Introduction}
Over the last few years, the study of free boundary minimal hypersurfaces (FBMH for short) has occupied a prominent place in differential geometry, especially the study of FBMH in the Euclidean unit ball $\B\subset\R^{n+1}$ (see \textit{e.g.} \cite{AmbCarSha18,AmCaSh,FraSch15,FraSch16,FraSch20,SmiSteTraZho,SmiZho} and the references therein).

In this paper, we deal with FBMH in $\R^{n+1}\setminus\B$ with compact boundary in 
$\p\B$, which are called \textsl{exterior free boundary minimal 
hypersurfaces}. These hypersurfaces are critical for the $n$-volume functional 
with respect to deformations that let the boundary on the unit sphere. For such critical 
points, the 
second order 
derivative of the volume functional is given by the so-called stability operator which here 
has a contribution from the boundary. In our situation, this contribution is nonnegative 
due to the concavity of the unit sphere with respect to its outside. An interesting 
question would be to understand the geometry and the topology of these hypersurfaces 
in terms of their indices. In the ball, this is the study made by L. Ambrozio, A. Carlotto and 
B. Sharp in \cite{AmCaSh}. A situation where non-compact FBMH have been studied is the 
case of Schwarzschild space: R. Montezuma 
\cite{Mont} and E. Barbosa and J.M. Espinar \cite{BarEsp} have looked at some properties of 
these hypersurfaces.

In $\R^{n+1}\setminus \B$, important examples of exterior FBMH are the \textsl{catenoidal hypersurfaces}, defined as 
exterior FBMH invariant by isometries fixing a straight line. In Section~\ref{sec:catenoid} 
we give a complete description of the catenoidal hypersurfaces and calculate their 
indices: some have index $0$ and others index $1$.

The aim of this paper is to prove two classification results for catenoidal hypersurfaces. The first classification theorem is the following. Definitions are given in Sections~\ref{sec:FBMH}~and~\ref{sec:stable}.

\begin{thm}\label{th:stab1}
Let $\Sigma$ be an exterior free boundary minimal hypersurface in $\R^{n+1}\setminus\B$. Let us assume that $\Sigma$ is stable and has parallel embedded regular ends. Then $\Sigma$ is a catenoidal hypersurface.
\end{thm}

In order to prove Theorem~\ref{th:stab1}, we first prove that saying that $\Sigma$ is 
stable is equivalent to saying that there exists a positive Jacobi function $u$ on 
$\Sigma$ satisfying the Robin boundary condition $\p_\nu u+u=0$ on $\p\Sigma$. This 
is the content of Proposition~\ref{prop:stab}. Its proof is based on the proof of a classical 
stability characterization due to D. Fischer-Colbrie and R. Schoen~\cite[Theorem~1]{FCSc} 
for manifolds without boundary, where here we make use of the Harnack inequality for 
positive solutions to $\Delta u+qu=0$ on $\Sigma$ with Robin boundary condition 
$\p_\nu u+u=0$ on $\p\Sigma$ proved in Appendix~\ref{ap:harnack}. Second, we obtain 
a B\^ocher type theorem for positive Jacobi functions on regular minimal ends in 
$\R^{n+1}$ which, together with Proposition~\ref{prop:stab}, implies that $\Sigma$ is 
invariant by isometries fixing a straight line, in other words, $\Sigma$ is a catenoidal 
hypersurface.

As mentioned above, some of the catenoidal hypersurfaces have index $1$, so 
it 
would be interesting to know  if the index equal to $1$ implies that the hypersurface is 
catenoidal. When $n=2$, it would be a result similar to Lopez-Ros result \cite{LopRos} for 
boundaryless minimal surfaces. 
For 
example, it would be interesting to understand if a control on the index gives a control on 
the number of ends of the hypersurface. However, the positive contribution of the 
boundary to the stability operator seems to make this not an easy task.

The second classification theorem is the following.

\begin{thm}\label{th:one-ended}
Let $\Sigma$ be an exterior free boundary minimal hypersurface. If $\Sigma$ has one regular end, then $\Sigma$ is a catenoidal hypersurface.
\end{thm}

The proof of Theorem~\ref{th:one-ended} is based on a symmetrization procedure as in Schoen's paper~\cite{Sch2}.

The paper is organized as follows. In Section~\ref{sec:FBMH} we present some 
preliminaries on FBMH and prove Proposition~\ref{prop:stab}. In Section~\ref{sec:stable} 
we state and prove an auxiliary B\^ocher type result (Theorem~\ref{th:bocher}) and 
present the proof of Theorem~\ref{th:stab1}. In Section~\ref{sec:catenoid} we introduce 
the catenoidal hypersurfaces and give a complete description of them, including the 
calculation of their indices. In Section~\ref{sec:one-ended} we prove 
Theorem~\ref{th:one-ended}. Finally, in Appendix~\ref{ap:harnack} we present a 
proof of 
the Harnack inequality for positive functions satisfying a Robin type boundary condition.

\section{Free boundary minimal hypersurfaces in $\R^{n+1}\setminus\B$}\label{sec:FBMH}

Let $\Sigma$ be an $n$-manifold with compact boundary. We say that $F:\Sigma\to\R^{n+1}\setminus\B$ is an \textsl{exterior proper immersion} if $F$ is a proper immersion and $F(\Sigma)\cap\partial\B=F(\partial\Sigma)$. In this paper, we always consider such exterior proper immersions. We also assume that $\Sigma$ is orientable so that a unit normal $N$ is well defined along $F$. Besides, we will often identify $\Sigma$ with its image $F(\Sigma)$ and just say that $\Sigma$ is an \textsl{exterior hypersurface}.

We will consider exterior hypersurfaces that are critical for the $n$-volume functional with respect to any deformations keeping the boundary on $\partial\B$. Such a hypersurface $\Sigma$ has vanishing mean curvature and meets $\partial\B$ orthogonally: we call them \textsl{exterior free boundary minimal hypersurfaces}.

Basic examples are given by cones over minimal hypersurfaces $S\subset\S^n=\partial\B$:
\[
\Sigma=\{tp\in\R^{n+1};p\in S\text{ and }t\ge1\}.
\]
One can also consider exterior FBMH that are invariant by isometries fixing a straight line: they are called \textsl{catenoidal hypersurfaces}. Their complete description is given in Section~\ref{sec:catenoid}.

Let $\Sigma$ be an exterior free boundary minimal hypersurface. The free boundary condition implies that, at $P\in\partial\Sigma$, the outgoing unit normal $\nu(P)=-P$ is a principal direction of the second fundamental form $B$ of $\Sigma$. Indeed, for $T\in T\partial\Sigma$, we have
\begin{align}\label{eq:principal}
B(\nu,T)=(D_T\nu,N)=-B_{\partial\B}(T,N)=-(T,N)=0,
\end{align}
where $D$ is the covariant derivative in $\R^{n+1}$ and $B_{\partial\B}$ is the second fundamental form of $\partial\B$.

\begin{figure}[!ht]
\centering
\includegraphics[scale=.975,trim={{.125\textwidth} 0 {.125\textwidth} 0},clip]{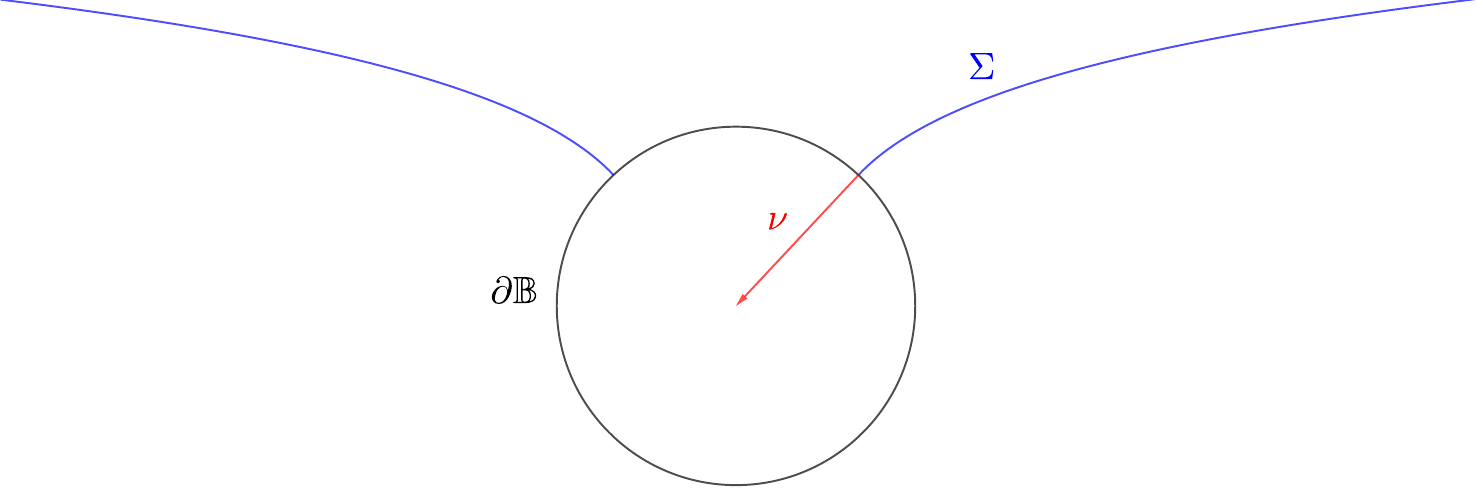}
\caption{Exterior FBMH}
\end{figure}

As in the boundaryless case, we also have a monotonicity formula for exterior free boundary minimal hypersurfaces:
\begin{align}\label{eq:monotonicity}
\frac{|\Sigma\cap B_R|}{R^n}-\left(1-\frac{1}{R^n}\right)\frac{|\p\Sigma|}{n}=\int_{\Sigma\cap B_R}\frac{|X^\perp|^2}{|X|^{n+2}},
\end{align}
where $B_R$ is the Euclidean ball centered at the origin of radius $R\ge1$. In fact, let $v(r)=|\Sigma\cap B_r|$, $r>1$. It follows from coarea formula that
\begin{align*}
\frac{d}{dr}v(r)=\int_{\Sigma\cap\p B_r}\frac{1}{|\nabla^\Sigma d|},
\end{align*}
where $d(X)=|X|$ is the distance function to the origin. On the other hand, because $\Sigma$ is minimal, $\div_\Sigma(X^\top)=n$. Therefore,
\begin{align*}
nv(r)=\int_{\Sigma\cap B_r}\div_\Sigma(X^\top)=\int_{\p\Sigma}(X^\top,\nu)+\int_{\Sigma\cap\p B_r}(X^\top,\nu)=-|\p\Sigma|+\int_{\Sigma\cap\p B_r}|X^\top|.
\end{align*}
This gives
\begin{align*}
\frac{d}{dr}\left(\frac{v(r)}{r^n}\right)=\frac{|\p\Sigma|}{r^{n+1}}+\int_{\Sigma\cap\p B_r}\frac{1}{r^n}\left(\frac{1}{|\nabla^\Sigma d|}-\frac{|X^\top|}{r}\right)=\frac{|\p\Sigma|}{r^{n+1}}+\int_{\Sigma\cap\p B_r}\frac{1}{|\nabla^\Sigma d|}\left(\frac{|X^\perp|^2}{r^{n+2}}\right),
\end{align*}
where above we have used that $\nabla^\Sigma d=\frac{X^\top}{|X|}$ and $|X^\perp|^2=r^2-|X^\top|^2$. Thus, integrating last equation from 1 to $R$ and using coarea formula, we obtain \eqref{eq:monotonicity}. 

\subsection{The stability operator}

Let $\{F_t\}$ be a family of exterior proper immersions of $\Sigma$ such that $F_0(\Sigma)$ is free boundary minimal and $\partial_tF_t$ has compact support. Even if the volume of $F_t(\Sigma)$ is infinite its derivatives can be computed since the deformation has compact support. Then the first derivative of the $n$-volume functional vanishes at $t=0$ and the second derivative at $t=0$ can be computed in terms of the function $u=(\partial_t{F_t}|_{t=0},N)$ by
\[
\frac{d^2}{dt^2}\Vol(F_t(\Sigma))\big|_{t=0}=Q(u,u)=\int_\Sigma(|\nabla^\Sigma u|^2-\|B\|^2u^2)d\mu+\int_{\partial\Sigma}u^2ds,
\]
where $\nabla^\Sigma$ and $d\mu$ are the gradient and the $n$-volume measure on $\Sigma$, and $ds$ is the $(n-1)$-volume measure on $\partial\Sigma$, all with respect to the metric induced by $F_0$ (see \cite{AmCaSh}). After integration by parts, one has 
\[
Q(u,u)=-\int_\Sigma u(\Delta 
u+\|B\|^2u)d\mu+\int_{\partial\Sigma}u(u+\partial_\nu u)ds.
\]
So the quadratic form $Q$ is associated with the Jacobi operator defined by $\boL u=\Delta u+\|B\|^2 u$. Then, for any bounded domain $\Ome$ in $\Sigma$, we can consider the associated spectrum of $\boL$: a sequence of eigenvalues $\lambda_n\nearrow+\infty$ and a $L^2$-orthonormal sequence of functions $u_n$ on $\Ome$ such that
\[
\begin{cases}
\Delta u_n+\|B\|^2 u_n=-\lambda_n u_n&\text{on }\Ome,\\
\partial_\nu u_n+u_n=0&\text{on }\partial\Sigma\cap\Ome,\\
u_n=0&\text{on }\partial\Ome\setminus\partial\Sigma.
\end{cases}
\]
The number of negative eigenvalues is then called the \textsl{index of $Q$ on $\Ome$} and is denoted by $\Ind(Q,\Ome)$. If $(\Ome_n)$ is an increasing sequence of domains such that $\cup\Ome_n=\Sigma$, then the limit of the increasing sequence $(\Ind(Q,\Ome_n))$ is called the \textsl{index of $\Sigma$} and is denoted by $\Ind(\Sigma)$.

When $\Ind(\Sigma)=0$, we say that $\Sigma$ is \textsl{stable} and this is equivalent to $Q(u,u)\ge 0$ for any function $u$ with compact support on $\Sigma$. Actually, we have an alternative characterization of stability given by the following Fischer-Colbrie and Schoen type result. 

\begin{prop}\label{prop:stab}
Let $\Sigma$ be an exterior free boundary minimal hypersurface. Then $\Sigma$ is stable if and only if there exists a positive solution $u$ on $\Sigma$ to
\begin{equation}\label{eq:stab}
\begin{cases}
\Delta u+\|B\|^2 u=0&\text{on }\Sigma,\\
\partial_\nu u+u=0&\text{on }\partial\Sigma.\\
\end{cases}
\end{equation}
\end{prop}

\begin{proof}
The proof is very similar to that one of \cite[Theorem~1]{FCSc} by Fischer-Colbrie and Schoen. In fact, in order to adapt their proof, we just need to observe that, if $u$ is as in \eqref{eq:stab}, then $\partial_\nu\ln u=-1$ on $\p\Sigma$ (this is for the part $(iii)\implies (i)$) and use the Harnack inequality given in Proposition~\ref{prop:harnack} in Appendix~\ref{ap:harnack} (for the part $(ii)\implies (iii)$).
\end{proof}

When $n=2$, Fischer-Colbrie's result \cite[Theorem~2]{Fis} gives that an exterior free boundary minimal surface $\Sigma$ has finite index if and only if it has finite total curvature. One difference is that, in our case, the quadratic form $Q$ does not depend only on the Gauss map but also on the conformal factor along the boundary $\p\Sigma$. A second important point is that we assume $\partial\Sigma$ to be compact. For example, if $\Sigma$ is stable, we have a solution $u$ to \eqref{eq:stab} which can be lifted to the universal cover $\widetilde\Sigma$. This implies that the associated quadratic form on $\widetilde\Sigma$ is nonnegative. However, the universal cover may not have finite total curvature as we are going to see below (see Example~\ref{exam:universal_cover}). Actually, the universal cover is not properly immersed and thus it is not an exterior surface according to our definition.

\subsection{Regular ends}

The asymptotic of an exterior free boundary minimal hypersurface can be highly complicated. A simple asymptotic is given by \textsl{regular ends} introduced by Schoen in \cite{Sch2}.

In order to describe it, we split $P\in\R^{n+1}$ as $(X,z)\in\R^n\times\R$. Then an end $E$ of an exterior free boundary minimal hypersurface is said to be \textsl{regular} if, after an isometry, a representative of $E$ is given by the graph of a function $f$ of bounded gradient defined on $\{|X|\ge R\}$ with the following asymptotic:
\begin{gather}
\label{eq:regular}
f(X)=A\ln|X|+B+(C,X)|X|^{-2}+O(|X|^{-2})\quad\text{if }n=2,\\
\label{eq:regular2}
f(X)=B+A|X|^{-(n-2)}+(C,X)|X|^{-n}+O(|X|^{-n})\quad\text{if }n>2,
\end{gather}
where $A,B\in\R$ and $C\in\R^n$. We notice that the above estimate on $f$ implies similar estimates on its derivatives (see \cite{Sch2}). For example, one sees that $\nabla f(X)$ goes to $0$ as $|X|$ goes to $+\infty$ and, in particular, there is a well-defined unit normal at $\infty$ for such an end.

If $n=2$ and $\Sigma$ has finite total curvature, then $\Sigma$ is conformally equivalent to a compact Riemann surface with boundary minus a finite number of points. Moreover, a properly embedded annular end with finite total curvature is regular \cite[Proposition~1]{Sch2}.

In the case $3\le n\le 6$, following the arguments of J. Tysk \cite{Tys}, if we 
assume that $\Sigma$ has finite index and finite volume growth in the sense that 
$\lim_{R\to+\infty}R^{-n}|\Sigma\cap B_R|<+\infty$, then $\Sigma$ has finitely many 
ends, all of them being regular.

\section{Stable hypersurfaces}\label{sec:stable}

\subsection{A B\^ocher type result for the Jacobi operator}

In this section, we analyze the asymptotic behavior of positive Jacobi functions (\textit{i.e.} solutions to $\boL u=0$) on regular ends.

\begin{thm}\label{th:bocher}
Let $E$ be a regular minimal end in $\R^{n+1}$ (let $X\in\R^n$ be a coordinate 
associated to the end as in \eqref{eq:regular} and \eqref{eq:regular2}) and consider a 
positive Jacobi function 
$u$ on $E$. Then $u$ has the following asymptotic: there exist $A,B\in\R$ such that
\begin{gather*}
u(X)=A\ln|X|+B+v(X)\quad\text{if }n=2,\\
u(X)=A+B|X|^{-(n-2)}+v(X)\quad\text{if }n>2,
\end{gather*}
where  $v$ is 
such 
that the function $|X|^{n-1}v$ is $C^2$-bounded on $\R^n\setminus B_R$. Moreover, 
either $A>0$ or $A=0$ and $B>0$.
\end{thm}

\begin{proof}
Writing $X=e^tp$ with $t\in\R$ and $p\in\S^{n-1}$, a regular end can be parametrized by $[t_0,+\infty)\times\S^{n-1}$ with a metric $g$ having the asymptotic $g=e^{2t}(\delta +O(e^{-2(n-1)t}))$, where $\delta$ is the product metric on $\R\times\S^{n-1}$. Moreover, the second fundamental form can be estimated by $\|B\|^2=O(e^{-2nt})$. Thus the Jacobi operator can be computed as
\[
\Delta u+\|B\|^2 u=e^{-2t}\left(u_{tt}+(n-2)u_t+\Delta^\sigma u+M(u)\right),
\]
where $\Delta^\sigma$ is the Laplacian on $\S^{n-1}$ and $M(u)$ is a second order linear operator whose coefficients have $C^{0,\alpha}$-norm bounded by $Ce^{-2(n-1)t}$ for some constant $C>0$.

Therefore a Jacobi function $u$ satisfies
\begin{equation}\label{eq:bocher}
u_{tt}+(n-2)u_t+\Delta^\sigma u+M(u)=0,
\end{equation}
which is a uniformly elliptic equation on $[t_0,+\infty)\times\S^{n-1}$. As a consequence, by Harnack inequality (\cite[Corollary~8.21]{GiTr}), there is a constant $C>0$ such that, for any $p,q\in\S^{n-1}$ and $t,s\ge t_0+1$ with $|t-s|\le 1$, and any positive Jacobi function $u$, we have
\[
u(t,p)\le C u(s,q).
\]
By Schauder's elliptic estimates (\cite[Corollary~6.3]{GiTr}), we also have
\begin{equation}\label{eq:schauder}
\|u\|_{C^{2,\alpha}([t-\frac12,t+\frac12]\times\S^{n-1})}\le C\|u\|_{C^0([t-1,t+1]\times\S^{n-1})}\quad\mbox{for}\quad t\ge t_0+2.
\end{equation}
Let us define $\bar u(t)=\frac1{|\S^{n-1}|}\int_{\S^{n-1}}u(t,p)d\sigma$. By Harnack inequality, we obtain
\begin{equation}\label{eq:harnack}
\|u\|_{C^0([t-1,t+1]\times\S^{n-1})}\le C\min_{p\in\S^{n-1}}u(t,p)\le C\bar u(t).
\end{equation}
Then, combining with \eqref{eq:schauder}, there is a constant $C>0$ such that
\begin{equation}\label{eq:estimM1}
\|M(u)\|_{C^{0,\alpha}([t-\frac12,t+\frac12]\times\S^{n-1})}\le Ce^{-2(n-1)t}\bar u(t)
\end{equation}
and $\barre{M}(u)(t)=\frac1{|\S^{n-1}|}\int_{\S^{n-1}}M(u)(t,p)d\sigma$ satisfies
\begin{equation}\label{eq:estimM2}
\|\barre M(u)\|_{C^{0,\alpha}([t-\frac12,t+\frac12])}\le Ce^{-2(n-1)t}\bar u(t).
\end{equation}

By integrating \eqref{eq:bocher} over $\S^{n-1}$, we obtain that $\bar u$ solves $\bar 
u_{tt}+(n-2)\bar u_t+\barre M(u)=0$. Considering first the case $n>2$, let $a$ and $b$ 
be two functions such that
\[
\begin{pmatrix}
\bar u\\ \bar u'
\end{pmatrix}=a\begin{pmatrix}
1\\0
\end{pmatrix}+b\begin{pmatrix}
1\\2-n
\end{pmatrix}.
\]
Then we have the system
\[
\left\{
\begin{array}{rcl}
a'&=&-\frac1{n-2}\barre M(u)(t),\\
b'&=&-(n-2) b+\frac1{n-2}\barre M(u)(t).
\end{array}\right.
\]
Using the above equations, we obtain
\[
\partial_t\sqrt{a^2+b^2}\le C|\barre M(u)(t)|\le Ce^{-2(n-1)t}\bar u(t)\le 
Ce^{-2(n-1)t}\sqrt{a^2+b^2}. 
\]
Thus $\sqrt{a^2+b^2}$ and $\bar u$ stay bounded on $[t_0,+\infty)$. In particular, 
$|\barre M(u)(t)|\le Ce^{-2(n-1)t}$. Also,
\begin{align*}
\partial_t\left(e^{nt}\sqrt{a^2+b^2}\right)&\ge 
ne^{nt}\sqrt{a^2+b^2}-(n-2)e^{nt}\frac{b^2}{\sqrt{a^2+b^2}}-Ce^{-(n-2)t}\sqrt{a^2+b^2}\\
&\ge (2-Ce^{-2(n-1)t})e^{nt}\sqrt{a^2+b^2}\ge 0
\end{align*}
for $t$ sufficiently large. Therefore $e^{nt}\sqrt{a^2+b^2}$ cannot converge to $0$ at $t$ goes to $+\infty$. We can also solve the system to obtain
\[
\left\{
\begin{array}{rcl}
a&=&A+\int_t^{+\infty}\frac1{n-2}\barre M(u)(s)ds,\\
b&=&Be^{-(n-2)t}-e^{-(n-2)t}\int_t^{+\infty}\frac1{n-2}e^{(n-2)s}\barre M(u)(s)ds.
\end{array}\right.
\]
We notice that if $A=B=0$ then $\lim_{t\to +\infty} e^{nt}a=\lim_{t\to+\infty} e^{nt}b=0$, 
which is not possible. Then we can be sure that either $A$ or $B$ is nonzero.  As $\bar 
u=a+b$ is positive and $A=\lim_{t\to+\infty}(a+b)$, then either $A>0$ or $A=0$ and 
$B>0$. Observe that $\bar u-A-Be^{-(n-2)t}=O(e^{-2(n-1)t})$.

If $n=2$, we notice that
\begin{align*}
\partial_t\sqrt{\bar u^2+\bar u_t^2}&\le\frac{\bar u\bar u_t}{\sqrt{\bar u^2+\bar u_t^2}}+C|\barre{M}(u)|\le\frac12\sqrt{\bar u^2+\bar 
u_t^2}+Ce^{-2t}\bar u\\
&\le\left(\frac12+Ce^{-2t}\right)\sqrt{\bar u^2+\bar u_t^2}.
\end{align*}
Thus $\bar u=O(e^{\frac12 t})$ and then $\barre M(u)=O(e^{-\frac32 t})$. We also have
\[
\partial_t\left(e^{\frac34t}\sqrt{\bar u^2+\bar u_t^2}\right)\ge 
\left(\frac34-\frac12-Ce^{-2t}\right)e^{\frac34 t}\sqrt{\bar u^2+\bar u_t^2}\ge 0
\]
for $t$ sufficiently large. So $e^{\frac34t}\sqrt{\bar u^2+\bar u_t^2}$ cannot converge to $0$ as $t$ goes to $+\infty$. By integrating the equation on $\bar u$, one gets
\[
\bar u(t)=At+B-\int_t^{+\infty}\left(\int_s^{+\infty}\barre M(u)(r)dr\right)ds.
\]
If $A$ and $B$ vanish, then $\bar u,\bar u_t=O(e^{-\frac32t})$, which is not possible. 
Then either $A>0$ or $A=0$ and $B>0$. Notice that $\bar u-At-B=O(te^{-2t})$. In fact, 
last equation gives that $\bar u=O(t)$. Then, from \eqref{eq:estimM2} we have $\barre 
M(u)=O(te^{-2t})$. Therefore, using last equation again, we obtain $\bar 
u-At-B=O(te^{-2t})$.

In both cases, we have $M(u)=O(te^{-2(n-1)t})$.

Now, to conclude, we need to estimate $u-\bar u$. Let $v_i$ be a $L^2$-unit eigenfunction for the Laplace operator on $\S^{n-1}$ associated to a nonzero eigenvalue $\lambda$ (in particular, $\lambda\ge n-1$). Let $u_i=\int_{\S^{n-1}}uv_id\sigma$. Equation \eqref{eq:bocher} implies
\[
{u_i}_{tt}+(n-2){u_i}_t-\lambda u_i=-\int_{\S^{n-1}}M(u)v_id\sigma=f_i=O(te^{-2(n-1)t}).
\]
Observe that $\mu^2+(n-2)\mu-\lambda=0$ has two roots: $\mu_+\ge 1$ and $\mu_-\le -(n-1)$. Then, solving the above equation, we obtain
\begin{equation}\label{eq:u^i}
u_i(t)=e^{\mu_+ t}\left(a_i-\int_t^{+\infty}e^{-\mu_+s}\frac{f_i(s)}{\mu_+-\mu_-}ds\right)+e^{\mu_-t}\left(b_i-\int_{t_0}^te^{-\mu_-s}\frac{f_i(s)}{\mu_+-\mu_-}ds\right)
\end{equation}
for some $a_i,b_i\in\R$. Using \eqref{eq:harnack} and the fact that $\bar u=O(t)$, we have $u_i=O(t)$ and thus $a_i=0$. We also have
\[
b_i=e^{-\mu_- t_0}u_i(t_0)+e^{(\mu_+-\mu_-)t_0}\int_{t_0}^{+\infty}e^{-\mu_+ s}\frac{f_i(s)}{\mu_+-\mu_-}ds.
\]
Now, by Cauchy-Schwarz,
\begin{gather*}
\left(\int_t^{+\infty}e^{-\mu_+ s}f_i(s)ds\right)^2\le\frac1{\mu_+}e^{-\mu_+ t}\int_t^{+\infty} e^{-\mu_+s}f_i^2(s)ds,\\
\left(\int_{t_0}^te^{-\mu_-s}f_i(s)ds\right)^2\le\frac1{-\mu_-}(e^{-\mu_-t}-e^{-\mu_-t_0})\int_{t_0}^te^{-\mu_-s}f_i^2(s)ds.
\end{gather*}
Thus, by squaring \eqref{eq:u^i}, we obtain
\begin{align*}
u_i^2(t)\le16\bigg(&\frac{e^{\mu_+t}}{\mu_+(\mu_+-\mu_-)^2}\int_t^{+\infty}e^{-\mu_+ s} f_i^2(s)ds+e^{2\mu_-(t-t_0)}u_i^2(t_0)\\
&+\frac{e^{2\mu_-(t-t_0)+\mu_+t_0}}{\mu_+(\mu_+-\mu_-)^2}\int_{t_0}^{+\infty}e^{-\mu_+ s} f_i^2(s)ds+\frac{e^{\mu_-t}}{-\mu_-(\mu_+-\mu_-)^2}\int_{t_0}^te^{-\mu_-s}f_i^2(s)ds\bigg).
\end{align*}

Let us define 
\begin{gather*}
\widetilde U(t)=\int_{\S^{n-1}}(u(t,p)-\bar u(t))^2d\sigma,\\
\widetilde M(t)=\int_{\S^{n-1}}(M(u)(t,p)-\barre{M}(u)(t))^2d\sigma.
\end{gather*}
Using that $\mu_+\ge 1$ and $\mu_-\le-(n-1)$, we can sum the above inequalities with respect to $i$ to obtain
\begin{align*}
\widetilde U(t)\le16\bigg(&e^{\mu_+t}\int_t^{+\infty}e^{-\mu_+s}\widetilde M(s)ds+e^{2\mu_-(t-t_0)}\widetilde U(t_0)\\
&+e^{2\mu_-(t-t_0)+\mu_+t_0}\int_{t_0}^{+\infty}e^{-\mu_+s}\widetilde M(s)ds+e^{\mu_-t}\int_{t_0}^te^{-\mu_-s}\widetilde M(s)ds\bigg).
\end{align*}
Since $\widetilde M(t)=O(t^2e^{-4(n-1)t})$, it follows that
\[
\|u-\bar u\|_{L^2([t-1,t+2]\times\S^{n-1})}=O(e^{-(n-1)t}).
\]
Actually, $u-\bar u$ solves the equation
\[
z_{tt}+(n-2)z_t+\Delta^\sigma z+M(z)=\barre M(u)-M(\bar u).
\]
Then, combining the above $L^2$-estimate with \eqref{eq:estimM1} and \eqref{eq:estimM2}, Schauder's estimates give
\[
\|u-\bar u\|_{C^2([t,t+1]\times\S^{n-1})}=O(e^{-(n-1)t}).
\]
We have then proved that
\begin{gather*}
\|u-A-Be^{-(n-2)t}\|_{C^2([t,t+1]\times\S^{n-1})}=O(e^{-(n-1)t})\quad\text{if }n>2,\\
\|u-At-B\|_{C^2([t,t+1]\times\S^{n-1})}=O(e^{-t})\quad\text{if }n=2.
\end{gather*}
This gives the expected result after going back to the original coordinate system.
\end{proof}

\subsection{Classification of stable hypersurfaces}

If $\Sigma$ is an exterior free boundary minimal hypersurface with regular ends, the unit normal to $\Sigma$ has a well-defined limit at each end. Then we say that such a hypersurface has \textsl{parallel ends} if these limits coincide up to a sign. We notice that, if $\Sigma$ is embedded, then its ends are always parallel.

Now, we are going to use the above B\^ocher type theorem in order to give a classification 
of stable exterior FBMH with parallel regular ends.

\begin{proof}[Proof of Theorem~\ref{th:stab1}]
Consider the $(X,z)$ coordinate system on $\R^{n+1}$. After an isometry, we can assume that the unit normal to the ends of $\Sigma$ are given by $\pm e_z$. 

Now, let $M\in\boM_n(\R)$ be a skew-symmetric matrix and consider the Killing vector field $K(X,z)=MX$. Notice that $K$ generates isometries fixing the $z$-axis. Then the scalar product $u=(K,N)$ is a solution to $\Delta u+\|B\|^2u=0$ on $\Sigma$. Moreover, since $K$ is tangent to $\partial\B$, $u$ satisfies $\partial_\nu u+u=0$ on $\p\Sigma$.

Each end of $\Sigma$ can be parametrized by the graph of a function $f$ with the asymptotic given by \eqref{eq:regular} or \eqref{eq:regular2} (depending on $n$). In particular,
\[
N(X,f(X))=\pm\frac1{\sqrt{1+|\nabla f(X)|^2}}(-\nabla f(X)+e_z).
\]
So the asymptotic of $f$ gives that $u=O(|X|^{-(n-1)})$.

On the other hand, since $\Sigma$ is stable, there is a positive solution $v$ to \eqref{eq:stab}. The asymptotic of $v$ is given by Theorem~\ref{th:bocher}. As a consequence, we see that $u(X)/v(X)$ goes to $0$ as $|X|$ goes to~$+\infty$. Also, the function $w=u/v$ satisfies
\[
\begin{cases}
\Delta w+2(\nabla\ln v,\nabla w)=0&\text{on }\Sigma,\\
\partial_\nu w=0&\text{on }\partial\Sigma.
\end{cases}
\]
Therefore the maximum principle gives that $w=u/v$ is constant and thus equals zero. This implies that $u=0$ and then $\Sigma$ is invariant by the isometries generated by $K$. So $\Sigma$ is a catenoidal hypersurface.
\end{proof}

\section{Catenoidal hypersurfaces}\label{sec:catenoid}

Theorem~\ref{th:stab1} gives that stable hypersurfaces are invariant by isometries fixing an axis. In this section, we describe this kind of exterior free boundary minimal hypersurfaces $\Sigma$. We fix the axis to be $\R e_z$.

Let $\phi$ be a primitive of the function $r\mapsto (r^{2(n-1)}-1)^{-1/2}$ defined on $[1,+\infty)$. The hypersurface
\[
\Sigma=\{(X,z)\in\R^n\times\R;|X|\ge1\text{ and }z^2=\phi^2(|X|)\}
\]
is a minimal hypersurface invariant by isometries fixing $\R e_z$. Actually, any connected piece of a minimal hypersurface invariant by isometries fixing $\R e_z$ is a subset of $\lambda\Sigma+\mu e_z$ for some $\lambda,\mu\in\R$.

Half of $\Sigma$ can be parametrized by the map
\[
F:\begin{array}{ccc}
[1,+\infty)\times\S^{n-1}&\to&\R^{n+1}\\
(r,p)&\mapsto&(rp,\phi(r))
\end{array}.
\]

Given $\alpha\in(0,\frac\pi2)$, we look for a rotational exterior free boundary minimal hypersurface with boundary in $\{z=\sin\alpha\}$. Let $R_\alpha=(\sin\alpha)^{-1/(n-1)}$ be such that $\phi'(R_\alpha)=\tan\alpha$. Notice that $R_\alpha$ decreases with $\alpha$ from $+\infty$ to $1$. Let $\boC_\alpha$ be the hypersurface parametrized by
\[
F_\alpha:\begin{array}{ccc}
[R_\alpha,+\infty)\times\S^{n-1}&\to&\R^{n+1}\\
(r,p)&\mapsto&\lambda_\alpha(rp,\phi(r))+\mu_\alpha e_z
\end{array},
\]
where $\lambda_\alpha$ and $\mu_\alpha$ are chosen such that $\boC_\alpha$ has the expected boundary:
\begin{equation}\label{eq:lambdamu}
\begin{cases}
\lambda_\alpha=R_\alpha^{-1}\cos\alpha=(\sin\alpha)^{\frac{1}{n-1}}\cos\alpha,\\
\mu_\alpha=\sin\alpha-\lambda_\alpha\phi(R_\alpha).
\end{cases}
\end{equation}
The hypersurface $\boC_\alpha$ has free boundary because of the choice of $R_\alpha$. Thus, for any $\alpha\in[0,\frac\pi2)$, there is exactly one rotational exterior free boundary minimal hypersurface: $\boC_\alpha$ for $\alpha\neq 0$ and $\boC_0=\{z=0\}\setminus\B$ for $\alpha=0$.

Therefore $\boC_\alpha$ is an exterior free boundary minimal catenoidal hypersurface and any connected exterior free boundary minimal catenoidal hypersurface is the image of some $\boC_\alpha$ by a linear isometry of $\R^{n+1}$.

\begin{figure}[!ht]
\centering
\includegraphics[scale=.975]{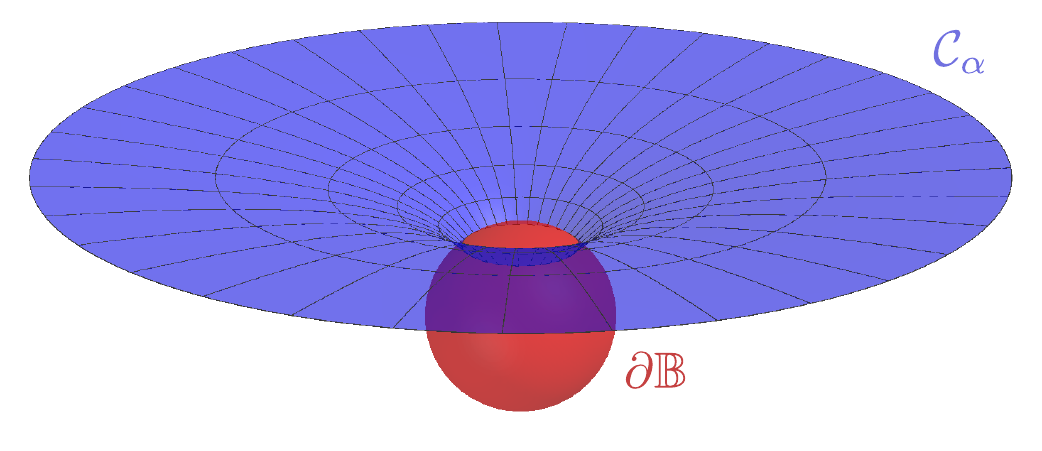}
\caption{Catenoidal hypersurface $\boC_\alpha$, $\alpha\approx0.92$}
\end{figure}

Observe that
\begin{equation}\label{eq:derlambda}
\partial_\alpha\lambda_\alpha=(\sin\alpha)^{-\frac{n-2}{n-1}}\left(\frac1{n-1}\cos^2\alpha-\sin^2\alpha\right).
\end{equation}

Let $\alpha_n=\arctan(\frac1{\sqrt{n-1}})$. It follows that, on $[0,\alpha_n]$, $\lambda_\alpha$ is increasing from $0$ to $\lambda_{\alpha_n}$ and, on $[\alpha_n,\frac\pi2)$, $\lambda_\alpha$ is decreasing up to $0$. Moreover, $\boC_\alpha$ converges to the hyperplane $\{z=1\}$ as $\alpha$ goes to $\frac\pi2$.

The stability properties of $\boC_\alpha$ is described by the following result, whose proof is inspired by the computation of the index of catenoids in $\R^{n+1}$ by L.-F. Tam and D. Zhou~\cite{TaZh}.

\begin{prop}
There exists $\bar\alpha_n\in[\alpha_n,\frac\pi2)$ such that $\boC_\alpha$ is stable for $\alpha\in[0,\bar\alpha_n)$, and $\boC_\alpha$ has index $1$ for $\alpha\in(\bar\alpha_n,\frac\pi2)$.

Actually, $\bar\alpha_2=\alpha_2=\frac\pi4$ and $\bar\alpha_n>\alpha_n$ for $n>2$.
\end{prop}

\begin{proof}
We first study some preliminary stability properties of $\boC_\alpha$. Notice that $\boC_\alpha$ is a graph over part of $\R^n$. Therefore $\boC_\alpha$ is stable as a graph with fixed boundary: the stability operator is nonnegative for any test functions that vanish on $\partial\boC_\alpha$.

Let us consider on $\R^n$ coordinates $(x,Y)\in\R\times\R^{n-1}$. Let $K(x,Y,z)=(-z,0,x)$ be the Killing vector field generating rotations around 
$\{x=0,z=0\}$ in $\R^{n+1}$. Then the scalar product $k=(K,N)$ defines on $\boC_\alpha$ a solution to \eqref{eq:stab}. The boundary condition comes from the fact that $K$ is tangent to $\partial\B$. Actually, one can compute $k$ in the $(r,p)$ coordinates. By \eqref{eq:lambdamu}, we have
\begin{align*}
k=\frac{1}{\sqrt{1+(\phi'(r))^2}}\left((\lambda_\alpha\phi(r)+\mu_\alpha)\phi'(r)+\lambda_\alpha r\right)p_x,
\end{align*}
where $p_x$ is the $x$ coordinate of $p\in\S^{n-1}\subset\R\times\R^{n-1}$. Observe that, by \eqref{eq:lambdamu},
\begin{align*}
(\lambda_\alpha\phi(r)+\mu_\alpha)\phi'(r)+\lambda_\alpha r&\ge(\lambda_\alpha\phi(R_\alpha)+\mu_\alpha)\phi'(r)+\lambda_\alpha R_\alpha\\
&=\sin\alpha\;\phi'(r)+\cos\alpha>0.
\end{align*}
Hence $k$ has constant sign when $p_x$ has constant sign. This implies that the half catenoidal hypersurfaces $\boC_\alpha\cap\{\pm x\ge 0\}$ are stable.

Let us now study the global stability of $\boC_\alpha$. We have a one-parameter family 
$\{\boC_\alpha\}$ of catenoidal hypersurfaces. Therefore the derivative with respect to 
$\alpha$ gives a deformation field whose scalar product with the unit normal to 
$\boC_\alpha$ is a function $u$ which solves \eqref{eq:stab}. In the 
$F_\alpha$ parametrization and for the upward pointing unit normal, $u$ can 
be computed as
\[
u=\frac1{\sqrt{1+\phi'(r)^2}}(-\partial_\alpha\lambda_\alpha r\phi'(r)+\partial_\alpha\lambda_\alpha\phi(r)+\partial_\alpha\mu_\alpha).
\]
So $u$ depends only on $r$ and is equal to $1$ on $\partial\boC_\alpha$, \textit{i.e.} at $r=R_\alpha$. In order to study the sign of $u$ close to 
$r=+\infty$, let us have a look on $\lambda_\alpha$. By \eqref{eq:derlambda}, we have
\[
\partial_\alpha^2\lambda_\alpha=-\frac{\cos\alpha\;(\sin\alpha)^{-\frac{2n-3}{n-1}} 
\left((n^2+n-2)\sin^2\alpha+(n-2)\cos^2\alpha\right)}{(n-1)^2}
\le 0.\]
Thus $\partial_\alpha\lambda_\alpha$ is decreasing.

If $n=2$, we have $\lim_{r\to+\infty}\phi(r)=+\infty$. Then, for $\alpha\neq\alpha_2=\frac\pi4$, we have $\lim_{r\to+\infty}u=\pm\infty$ depending on $\textrm{sign}(\partial_\alpha\lambda_\alpha)$: close to $r=+\infty$, $u$ is positive for $\alpha<\alpha_2$ and negative for $\alpha>\alpha_2$.

If $n>2$, we have $\lim_{r\to+\infty}\phi(r)<+\infty$ and, by \eqref{eq:lambdamu},
\begin{align*}
\lim_{r\to +\infty} u&=\partial_\alpha\mu_\alpha+\partial_\alpha\lambda_\alpha\lim_{r\to+\infty}\phi(r)\\
&=\frac n{n-1}\cos\alpha+\partial_\alpha\lambda_\alpha\left(\lim_{r\to+\infty}\phi(r)-\phi(R_\alpha)\right).
\end{align*}
This limit is positive when $\alpha\le \alpha_n$. Moreover, when $\alpha\ge 
\alpha_n$, the limit is decreasing with $\alpha$ and negative for $\alpha$ close to 
$\frac\pi2$. Then there exists $\bar\alpha_n>\alpha_n$ such that the limit is positive 
for $\alpha<\bar\alpha_n$ and negative for $\alpha>\bar\alpha_n$.

Thus, for $\alpha<\bar\alpha_n$, $u$ is positive on $\partial\boC_\alpha$ and close to the infinity. Therefore, if $u$ changes sign on $\boC_\alpha$, $\{u<0\}$ would be a precompact subdomain of $\boC_\alpha$ with $u=0$ on its boundary, but this would contradict the stability of $\boC_\alpha$ as a graph. Hence, for $\alpha<\bar\alpha_n$, $u$ is positive and then, by Proposition~\ref{prop:stab}, $\boC_\alpha$ is stable. $\boC_{\bar\alpha_n}$ is also stable as limit of stable minimal hypersurfaces.

When $\alpha>\bar\alpha_n$, $u$ changes sign on $\boC_\alpha$. Thus there is $A>R_\alpha$ such that $u$ is nonnegative on $[R_\alpha,A]\times\S^{n-1}$ and vanishes on $\{A\}\times\S^{n-1}$. This implies that $\boC_\alpha$ has index at least $1$. We notice that there is no value $B>A$ such that $u$ vanishes on $\{B\}\times\S^{n-1}$. Indeed, this would contradict that $[A,+\infty)\times\S^{n-1}$ is stable as a graph.

Let us now prove that, for $\alpha>\bar\alpha_n$, $\boC_\alpha$ has index $1$. If it is 
not the case, then there is $B>R_\alpha$ such that the Jacobi operator has index at least 
$2$ on $[R_\alpha,B]\times\S^{n-1}$. Let us consider $u_2$ the eigenfunction 
associated to the second eigenvalue $\lambda_2<0$ on 
$\boC_\alpha(B)=F_\alpha([R_\alpha,B]\times\S^{n-1})$: $u_2$ is a solution 
to 
\[
\begin{cases}
\Delta u_2+\|B\|^2u_2=-\lambda_2 u_2&\text{on }\boC_\alpha(B),\\
\partial_\nu u_2+u_2=0 &\text{on }\partial\boC_\alpha,\\
u_2=0&\text{on }r=B.\\
\end{cases}
\]
We are going to prove that $u_2$ depends only on the $r$ variable. As above, let us consider $(x,Y)$ coordinates on $\R^n$ and let $S$ be the symmetry of $\R^{n+1}$ with respect to $\{x=0\}$. $\boC_\alpha(B)$ is invariant by $S$ and then we can consider on it the function $v$ defined by $v(p)=u_2(p)-u_2(S(p))$. $v$ is then a solution to 
\[
\begin{cases}
\Delta v+\|B\|^2v=-\lambda_2 v&\text{on }\boC_\alpha(B),\\
\partial_\nu v+v=0 &\text{on }\partial\boC_\alpha,\\
v=0&\text{on }r=B.\\
\end{cases}
\]
Moreover, we have $v=0$ on $\boC_\alpha(B)\cap\{x=0\}$. If $v\neq 0$, this implies that $\boC_\alpha(B)\cap\{x\ge 0\}$ is unstable since $\lambda_2<0$, which contradicts the stability of $\boC_\alpha\cap\{x\ge0\}$. So $v\equiv0$ and $u_2$ is invariant by $S$. Changing the choice of the $x$ coordinate, we obtain that $u_2$ is invariant by isometries fixing the $z$-axis and then depends only on $r$.

As $u_2$ is associated to the second eigenvalue, $u_2$ must change sign. Then there is $C\in (R_\alpha,B)$ such that $u_2=0$ on $\{r=C\}$. As $\lambda_2<0$, this implies that $\{C\le r\le B\}$ is unstable, which contradicts the stability of $\boC_\alpha$ as a graph. Hence $\boC_\alpha$ has index $1$ for $\alpha>\bar\alpha_n$.
\end{proof}

\begin{exam}\label{exam:universal_cover}
Consider the universal cover of $\boC_\alpha$ for $n=2$:
\begin{align*}
F_\alpha:\begin{array}{ccc}
[R_\alpha,+\infty)\times\R&\to&\boC_\alpha\subset\R^3\\
(r,\theta)&\mapsto&\lambda_\alpha(r\cos\theta,r\sin\theta,\phi(r))+\mu_\alpha e_z
\end{array}.
\end{align*}
Straightforward computations give that the area element and the Gaussian curvature of $F_\alpha$ are given by
\begin{gather*}
d\mu_\alpha=\lambda_\alpha^2r(1+(\phi')^2)^{\frac12}\,dr d\theta,\\
K_\alpha=\frac{\phi'\phi''}{\lambda_\alpha^2r(1+(\phi')^2)^2}.
\end{gather*}
Therefore,
\begin{gather*}
K_\alpha d\mu_\alpha=\frac{\phi'\phi''}{(1+(\phi')^2)^{\frac32}}\,drd\theta=-\frac{drd\theta}{r^2\sqrt{r^2-1}}.
\end{gather*}
Thus, the total curvature of $\boC_\alpha$ is given by
\begin{align*}
2\int_0^{2\pi}\left(\int_{R_\alpha}^{+\infty}\frac{dr}{r^2\sqrt{r^2-1}}\right)d\theta=4\pi\int_{R_\alpha}^{+\infty}\frac{dr}{r^2\sqrt{r^2-1}}=4\pi(1-\cos\alpha),
\end{align*}
while the total curvature of its universal cover is infinite. For $\alpha\le 
\frac\pi4$, this example shows 
that when the boundary is not compact, even if the stability operator is nonnegative, the 
total curvature can be infinite.
\end{exam}

\section{Classification of one-ended examples}\label{sec:one-ended}

This section is devoted to the proof of Theorem \ref{th:one-ended}. The idea of the proof is based on a symmetrization procedure as in Schoen's paper~\cite{Sch2}.

After a rotation, we can assume that the end of $\Sigma$ is the graph of a function $f$ over the outside of a compact set with the following asymptotic:
$$
f(X)=A\ln|X|+B+O(|X|^{-1})\quad\text{if }n=2,
$$
with $A\ge0$ and, if $A=0$, $B\ge0$, and
$$
f(X)=B+A|X|^{-(n-2)}+O(|X|^{-(n-1)})\quad\text{if }n>2,
$$
with $B\ge0$.

The first step consists in proving that either $\Sigma=\{z=0\}\setminus\B=\boC_0$ or $A>0$ and $\Sigma\subset\{z>\eps\}$ for some $\eps>0$.

Observe that $\p\Sigma\subset\{z\ge-2\}$ and $\Sigma\setminus K\subset\{z\ge-2\}$ for some compact set $K\subset\R^{n+1}$. Then, by the maximum principle, $\Sigma\subset\{z\ge-2\}$. In fact, for each $t<0$, we have $\Sigma\setminus K\subset\{z\ge t\}$ (for a possibly different compact set $K$). Therefore, if $\Sigma\cap\{z<0\}\neq\varnothing$, we can start from $t=-2$ and let $t<0$ increase up to finding a first contact point in $\Sigma\cap\{z=t_0\}$ for some $t_0<0$. We notice that, since $\Sigma$ is free boundary, the first contact point cannot be at $\p\Sigma$. Then the maximum principle can be applied at the first contact point in order to guarantee that $\Sigma=\{z=t_0\}\setminus\B$, which is not free boundary.

This shows that $\Sigma\subset\{z\ge0\}$. Then either $\p\Sigma\subset\{z>0\}$ or $\p\Sigma$ has a point in $\{z=0\}$ and the boundary maximum principle can be applied so that $\Sigma=\{z=0\}\setminus\B$.

If $\partial\Sigma\subset\{z>0\}$, then we see that 
\[
-\int_{\partial\Sigma}\partial_\nu z=\int_{\partial\Sigma}(P,e_z)>0.
\]
Since $z$ is harmonic on $\Sigma$, using the asymptotic of $f$ we obtain that
\begin{align*}
0<\int_{\Sigma\cap\{|X|=R\}}\partial_\nu z 
&=\int_{\Sigma\cap\{|X|=R\}}(\nu,e_z)\\
&=\int_{\S^{n-1}}\left(-(n-2)A\frac1{R^{n-1}} 
R^{n-1}+O(R^{-1})\right)d\sigma\\
&=-(n-2)|\S^{n-1}| A+o(1)
\end{align*}
for $n>2$. The same estimate gives that $0<2\pi A+o(1)$ for $n=2$. Therefore, if $n=2$, 
we have $A>0$ and this implies that $f(X)>1$ for $|X|$ sufficiently large. If $n>2$, we 
have $A<0$ and, since $f\ge 0$, this implies that $B>0$ and $f(X)>\frac B2$ for $|X|$ 
sufficiently large. In any case, we obtain that $\Sigma\subset\{z\ge\eps\}$ for small 
positive $\eps$ since there cannot be any first contact point with $\{z=t\}$ for $0\le 
t\le\eps$. This ends the first step.

We fix a $(x,Y)$ coordinate system in $\R^n=\R\times\R^{n-1}$. We want to prove that $\Sigma$ is symmetric with respect to $\{x=0\}$. In order to do this, we are going to follow a symmetrization procedure.

Given $\theta\in[0,\frac\pi 2]$, let $\Pi_\theta$ be the hyperplane of equation $-x\sin\theta+z\cos\theta=0$, $S_\theta$ be the symmetry with respect to $\Pi_\theta$, and $\Sigma_\theta^-=\Sigma\cap\{-x\sin\theta+z\cos\theta\le0\}$. If $B_\rho$ is the ball centered at the origin of radius $\rho>0$, we notice that $S_\theta(B_\rho)=B_\rho$.

\begin{lem}\label{lem:sym}
Given $\theta\in[0,\frac\pi2)$, there exists $\rho_\theta>0$ such that, outside $B_{\rho_\theta}$, $S_\theta(\Sigma_\theta^-)$ is above $\Sigma$. Moreover, $\rho_\theta$ can be chosen as an increasing function of $\theta$.
\end{lem}

\begin{proof}
In $\R^2$, let $p=(a\cos\alpha,a\sin\alpha)$ be a point with $a>0$ and 
$0\le\alpha\le\theta$ and $R_t$ be the rotation of angle $t$. If $0\le 
t\le2(\theta-\alpha)$, then the angle between $\overrightarrow{pR_t(p)}$ and the 
vertical $z$-axis is $\alpha-\frac t2$ and then at most $\theta$ (see 
Figure~\ref{fig3}).

\begin{figure}[!ht]
\centering
\includegraphics[scale=.975]{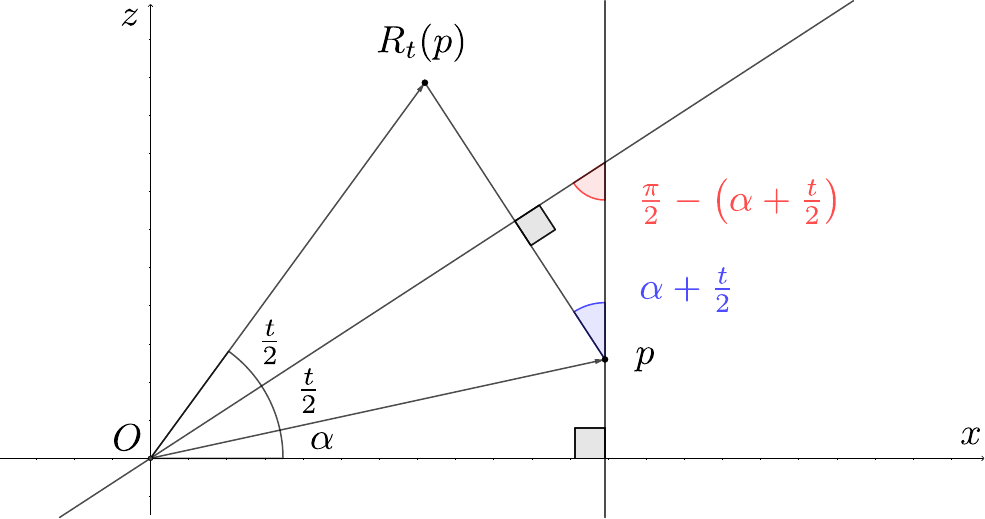}
\caption{Angle between $\protect\overrightarrow{pR_t(p)}$ and the $z$-axis \label{fig3}}
\end{figure}

Because of the asymptotic of $\Sigma$, the intersection $\Sigma\cap\partial B_\rho$ can be parametrized by $\partial B_\rho\cap\{z=0\}$ in the following way: there is a function $g$ such that
\[
\Sigma\cap\partial 
B_\rho=\big\{\big((1-\frac{g^2(X)}{|X|^2})^{1/2}X,g(X)\big);X\in\partial B_\rho\cap\{z=0\}\big\},
\]
where $g$ satisfies
\begin{gather*}
g(X)=A\ln|X|+B+O(|X|^{-1})\quad\text{if }n=2,\\
g(X)=B+A|X|^{-(n-2)}+O(|X|^{-(n-1)})\quad\text{if }n>2,
\end{gather*}
and
\begin{equation}\label{eq:grad}
dg(X)(V)=O(|X|^{-2})|V|
\end{equation} 
for any vector $V$ tangent to $\partial B_{|X|}$.

In the $(x,Y,z)$ coordinates, we can extend the rotation $R_t$ to $\R^{n+1}$ by fixing the $Y$ coordinates. Let $\rho$ be large and $P\in\Sigma_\theta^-\cap\partial B_\rho$. We can write $P=\big((1-\frac{g^2(X)}{|X|^2})^{1/2}X,g(X)\big)$ for some $X=(x,Y)$. Then $S_\theta(P)$ is above $\Sigma$ if $R_t(P)$ does not meet $\Sigma$ for $0<t\le 2(\theta-\alpha)$, where $0\le\alpha\le\theta$ is such that
\[
\big((1-\frac{g^2(X)}{|X|^2})^{1/2}x,g(X)\big)=a(\cos\alpha,\sin\alpha)
\]
for some $a>0$ (see Figure~\ref{fig4}). In particular, $S_\theta(P)=R_{2(\theta-\alpha)}(P)$. We notice that $R_t(P)$ belongs to $\partial B_\rho$. Then, if $R_t(P)\in 
\Sigma$, we must have
\[
R_t(P)=\big((1-\frac{g^2(X')}{|X'|^2})^{1/2}X',g(X')\big)
\]
for some $X'\in\partial B_\rho\cap\{z=0\}$. 

\begin{figure}[!ht]
\centering
\includegraphics[scale=.975,trim={{.125\textwidth} 0 {.125\textwidth} 0},clip]{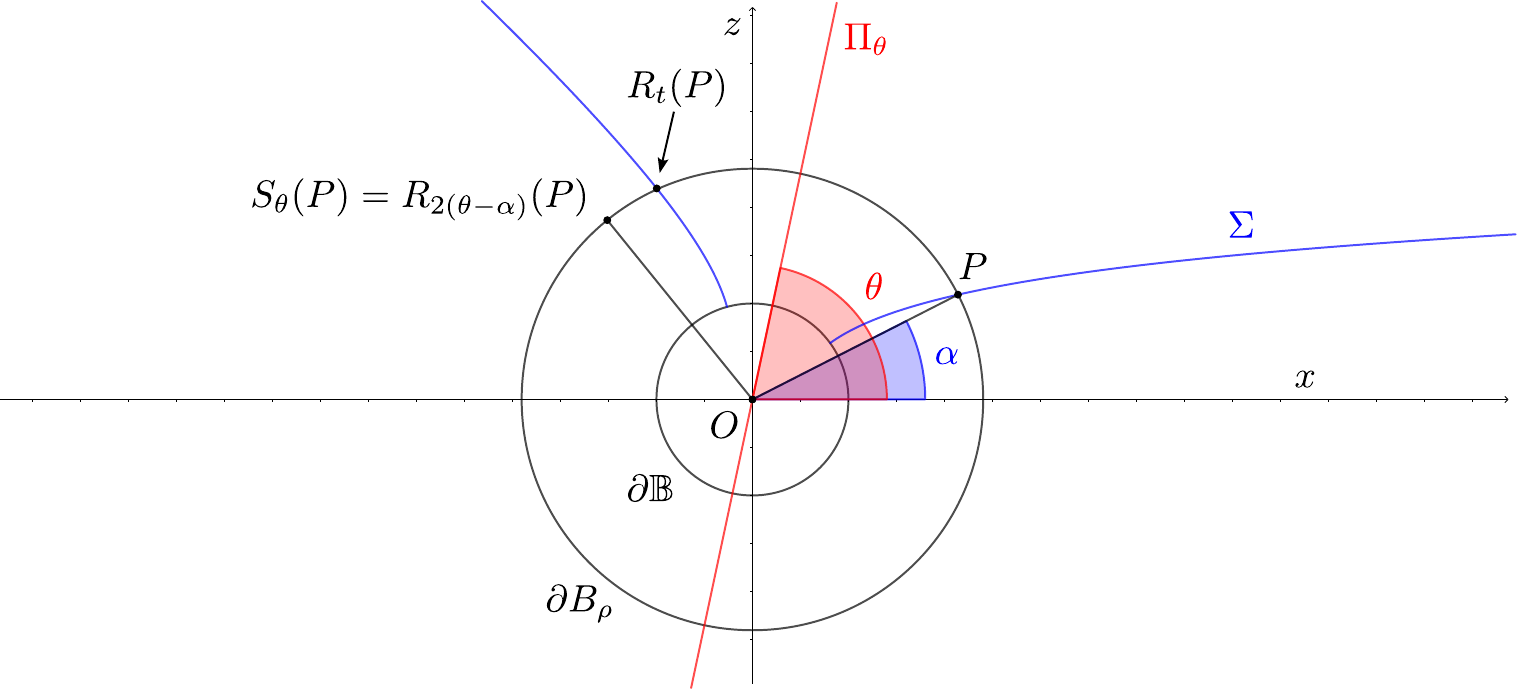}
\caption{Configuration if $S_\theta(P)$ is not above $\Sigma$ \label{fig4}}
\end{figure}

As a consequence, by integrating \eqref{eq:grad} along $\partial B_\rho\cap\{z=0\}$, we have
\begin{align*}
|g(X)-g(X')|\le C\rho^{-2}|X-X'|\le C'\rho^{-2}\Big|(1-\frac{g^2(X)}{|X|^2})^{1/2}X-(1-\frac{g^2(X')}{|X'|^2})^{1/2}X'\Big|.
\end{align*}
This implies that $\overrightarrow{PR_t(P)}$ makes an angle less than $C''\rho^{-2}$ with the horizontal plane $\{z=0\}$. Therefore, if $\rho_\theta$ is chosen such that this angle is less than $\frac\pi2-\theta$, we obtain a contradiction with $R_t(P)\in\Sigma$ and the lemma is proved.
\end{proof}

We are now ready to finish the proof of Theorem \ref{th:one-ended}. Let 
\[
T=\big\{\theta\in[0,\frac\pi2];S_\beta(\Sigma_\beta^-)\text{ is above }\Sigma\text{ for all }\beta\in[0,\theta]\big\}.
\]
Since $\Sigma\subset\{z>0\}$, we may choose $\theta_0>0$ small enough such that $\Sigma_{\theta_0}^-\subset\R^{n+1}\setminus B_{\rho_{\theta_0}}$. By Lemma~\ref{lem:sym}, we have $[0,\theta_0]\subset T$. The set $T$ is then a closed interval of the form $[0,\theta_1]$. Let us notice that when we symmetrize with respect to $\Pi_\theta$, the image of a point on $\partial B_\rho$ stays on $\partial B_\rho$, so that points in $\partial\Sigma$ cannot be sent to interior points of $\Sigma$ and interior points of $\Sigma$ cannot be sent to $\partial\Sigma$. 

Then, if $\theta_1<\frac\pi2$, by Lemma~\ref{lem:sym}, there is a point $P\in\Sigma_{\theta_1}^-$ such that one of the following occurs:
\begin{itemize}
\item $P\notin\Pi_{\theta_1}$, $S_{\theta_1}(P)\in\Sigma$ and $S_{\theta_1}(\Sigma_{\theta_1}^-)$ is on one side of $\Sigma$.
\item $P\in\Pi_{\theta_1}$, $\Sigma$ is orthogonal to $\Pi_{\theta_1}$ at $P$ and $S_{\theta_1}(\Sigma_{\theta_1}^-)$ is on one side of $\Sigma$.
\end{itemize}

In the first case, if $P\notin\partial\Sigma$, then maximum principle gives that $\Sigma$ is symmetric with respect to $\Pi_{\theta_1}$. If $P\in\partial\Sigma$, then, by free boundary hypothesis, $\Sigma$ and $S_{\theta_1}(\Sigma_{\theta_1}^-)$ are normal to $\partial\B$ and thus tangent, since $S_{\theta_1}(\Sigma_{\theta_1}^-)$ is on one side of $\Sigma$. As a consequence, the boundary maximum principle implies 
that $\Sigma$ is symmetric with respect to $\Pi_{\theta_1}$.

In the second case, if $P\notin\partial\Sigma$, the boundary maximum principle implies that $\Sigma$ is symmetric with respect to $\Pi_{\theta_1}$. If $P\in\partial\Sigma$, then as $S_{\theta_1}(\Sigma_{\theta_1}^-)$ is on one side of $\Sigma$, we can locally parametrize $\Sigma$ and 
$S_{\theta_1}(\Sigma_{\theta_1}^-)$ over a quarter of the tangent plane $T_P\Sigma$ by two functions $u$ and $v$ such that $u\le v$, $u(P)=v(P)$, and $\nabla u(P)=\nabla v(P)$. Moreover, at $P$, the tangent vector $P$ is an eigenvector of the second fundamental form of $\Sigma$ (see \eqref{eq:principal}). Thus the Hessian of $u$ and $v$ at $P$ coincide. So, applying Serrin's corner maximum principle \cite{Ser2} to $v-u$, we obtain that $u\equiv v$ and $\Sigma$ is symmetric with respect to $\Pi_{\theta_1}$.

In any case, we obtain that $\Pi_{\theta_1}$ is a plane of symmetry of $\Sigma$, which is not possible by Lemma~\ref{lem:sym}. This gives that 
$\theta_1=\frac\pi2$. Then $S_{\pi/2}(\Sigma\cap\{x\ge 0\})$ is above $\Sigma$. The same argument gives that $S_{\pi/2}(\Sigma\cap\{x\le 0\})$ is above $\Sigma$. As a consequence, $\Sigma$ is symmetric with respect to $\{x=0\}$.

By changing the coordinate system, we obtain that $\Sigma$ is symmetric with respect to any vertical hyperplane passing through the origin and then invariant by rotation around the vertical $z$-axis: $\Sigma$ is a catenoidal surface.

\appendix

\section{Harnack inequality}\label{ap:harnack}

In this paper, we are considering solutions $u$ to some elliptic equations on $\Sigma$ under the Robin boundary condition $\partial_\nu u+u=0$. Elliptic regularity theory for this condition can be found in \cite[Theorem~2.4.2.6]{Grs}. Besides, one can also remark that, if $d$ is a smooth function on $\Sigma$ with $\partial_\nu d=1$ (for example $-d$ could be the distance function to $\partial\Sigma$) and $v=e^d u$, then $\partial_\nu v=(\partial_\nu d)e^du+e^d\partial_\nu u=0$ and $v$ solves some elliptic equation. So results for Neumann boundary data can be translated to the Robin boundary condition.

In the proof of Proposition~\ref{prop:stab}, we use a Harnack inequality up to the boundary that can be derived from the following one.

\begin{prop}
Let $\Sigma$ be a Riemannian manifold with compact boundary and $u$ be a positive solution to
\[
\begin{cases}
\Delta u+(X,\nabla u)+qu=0&\text{ on }\Sigma,\\
\partial_\nu u=0&\text{ on }\partial\Sigma,
\end{cases}
\]
where $X$ is smooth vectorfield and $q$ a smooth function.
Then, given a compact domain $U\subset\Sigma$, 
there exists a constant $C>0$ (not depending on $u$) such that, for any 
$p,q\in U$, we have
\[
\frac{u(p)}{u(q)}\le C.
\]
\end{prop}

No such statement seems to appear in the literature. Similar results appear in \cite{Che2,Wan}, however they are not directly applicable here because of certain hypotheses.

\begin{proof}
In order to prove such an estimate, it is enough to prove an upper bound on $|\nabla\ln u|$. Let $v=\ln u$. We have
\begin{equation}\label{eq:harnack1}
\Delta v=\frac{\Delta u}u-\frac{|\nabla u|^2}{u^2}=-q-(X,\nabla v)-|\nabla 
v|^2.
\end{equation}
Now, let $w=|\nabla v|^2$ and consider a nonnegative function $\phi$ with compact support such that $\phi=1$ on $\partial\Sigma$ and $\phi\ge 1$ on $U$. Let consider $p$ a point of maximum of $\phi w$. We notice that $\partial_\nu v=0$, so $\nabla v$ is tangent to $\partial\Sigma$. Thus
\begin{align*}
\partial_\nu(\phi w)&=w\partial_\nu\phi+2\phi(\nabla_\nu\nabla v,\nabla v)\\
&=w\partial_\nu\phi+2\phi (\nabla_{\nabla v}\nabla v,\nu)\\
&=w\partial_\nu\phi+2\phi B_{\partial\Sigma}(\nabla v,\nabla v)\\
&\le w(\partial_\nu\phi+2H),
\end{align*}
where $B_{\partial\Sigma}$ is the second fundamental form of $\partial\Sigma$ and $H$ is an upper bound for the principal curvatures of $\partial\Sigma$. So, choosing $\phi$ such that $\partial_\nu\phi+2H<0$, we can ensure that $\partial_\nu(\phi w)$ is negative. Thus the maximum 
cannot be on the boundary of $\Sigma$. Let us compute $\Delta(\phi w)$ by using Bochner formula:
\begin{align*}
\Delta(\phi w)&=w\Delta \phi+2(\nabla \phi,\nabla w)+\phi\Delta(|\nabla v|^2)\\
&=w\Delta \phi+2(\nabla \phi,\nabla w)+2\phi\left[(\nabla v,\nabla \Delta 
v)+|\nabla^2 v|^2+\Ric(\nabla v,\nabla v)\right]\\
&\ge w(\Delta \phi-2K\phi)+2(\nabla \phi,\nabla 
w)+2\phi|\nabla^2v|^2+2\phi(\nabla v,\nabla(-q-(X,\nabla v)-w)),
\end{align*}
where $-K$ is a lower bound on the Ricci tensor. We also have 
\begin{align*}
(\nabla v,\nabla(X,\nabla v))&=(\nabla_{\nabla v}X,\nabla v)+(X,\nabla_{\nabla 
v}\nabla v)\\
&\le Kw+(\nabla v,\nabla_X\nabla v)\\
&=Kw+\frac12(X,\nabla w),
\end{align*}
where $K$ is also chosen to be an upper bound for the tensor $(\nabla_{\displaystyle\cdot}X,\cdot)$.

At $p$, we have $0=w\nabla \phi+\phi\nabla w$, that is, $\nabla w=-w\frac{\nabla 
\phi}\phi$. Thus
\begin{align*}
\Delta(\phi w)&\ge w(\Delta\phi-4K\phi) 
+2\phi|\nabla^2v|^2-2\phi(\nabla v,\nabla q)\\
&\qquad\qquad-2\frac{|\nabla\phi|^2}\phi w+w(X,\nabla \phi)+2w(\nabla v,\nabla \phi).
\end{align*}

At $p$, $\Delta(\phi w)\le 0$, and so
\begin{equation}\label{eq:harnack2}
0\ge2\phi|\nabla^2v|^2+ w(\Delta\phi-4K\phi-2\frac{|\nabla 
\phi|^2}\phi+(X,\nabla \phi))-2|\nabla\phi|w^{3/2}-2\phi|\nabla q|w^{1/2}.
\end{equation}

We have $|\nabla^2 v|\ge \frac1{\sqrt n}|\Delta v|$. Then, by \eqref{eq:harnack1}, we have
\[
|\nabla^2 v|^2\ge \frac1 n|\Delta v|^2=\frac1n (w^2-aw^{3/2}-bw-cw^{1/2}-d)
\]
for some constants $a$, $b$, $c$ and $d$ which depend only on $X$ and $q$. Thus, combining with \eqref{eq:harnack2} and multiplying by $\phi(p)$, we obtain at $p$,
\[
0\ge \frac2n(\phi w)^2-A(\phi w)^{3/2}-B(\phi w)-C(\phi w)^{1/2}-D
\]
for some constants $A$, $B$, $C$ and $D$ which depend only on $K$, $X$, $q$ and $\phi$. This implies that $\phi(p) w(p)\le M$ for some constant $M=M(A,B,C,D)$ and thus 
$w(q)\le M$ for $q\in U$, since $\phi\ge 1$ on $U$.
\end{proof}

As a consequence, we have the following Harnack inequality for the Robin boundary condition.

\begin{prop}\label{prop:harnack}
Let $\Sigma$ be a Riemannian manifold with compact boundary and $u$ be a positive solution to
\[
\begin{cases}
\Delta u+qu=0&\text{ on }\Sigma,\\
\partial_\nu u+u=0&\text{ on }\partial\Sigma,
\end{cases}
\]
where $q$ is a smooth function.
Then, given a compact domain $U\subset\Sigma$, there exists a constant $C>0$ 
(not depending on $u$) such that, for any $p,q\in U$, we have
\[
\frac{u(p)}{u(q)}\le C.
\]
\end{prop}

\begin{proof}
As explained above, let $d$ be a smooth function on $\Sigma$ such that $\partial_\nu d=1$ and $v=e^{d}u$. We then have 
$\partial_\nu v=(\partial_\nu d)e^du+e^d\partial_\nu u=0$. We also have 
\begin{align*}
\Delta v&=e^d(u|\nabla d|^2+u\Delta d+2(\nabla u,\nabla d)+d\Delta u)\\
&=2(\nabla v,\nabla d)+(\Delta d-|\nabla d|^2-dq)v.
\end{align*}
So the above proposition applies to $v$ and gives the expected result for $u$.
\end{proof}

\bibliographystyle{amsplain}
\bibliography{bibliography.bib}

\end{document}